\documentclass[reqno]{amsart}
\usepackage{amssymb}

\newtheorem{theorem}{Theorem}[section]
\newtheorem{proposition}[theorem]{Proposition}
\newtheorem{lemma}[theorem]{Lemma}
\newtheorem{corollary}[theorem]{Corollary}

\theoremstyle{definition}
\newtheorem*{definition}{Definition}

\theoremstyle{remark}

\numberwithin{equation}{section}

\begin{document}

\title[A Generalized Macdonald Operator]
{A Generalized Macdonald Operator}

\author{J.F.  van Diejen}
\author{E. Emsiz}
\address{
Instituto de Matem\'atica y F\'{\i}sica, Universidad de Talca,
Casilla 747, Talca, Chile}
\email{diejen@inst-mat.utalca.cl}
\email{eemsiz@inst-mat.utalca.cl}


\thanks{Work supported in part by the {\em Fondo Nacional de Desarrollo
Cient\'{\i}fico y Tecnol\'ogico (FONDECYT)} Grants \# 1090118 and
3080006, by the {\em Anillo ACT56 `Reticulados y Simetr\'{\i}as'}
financed by the  {\em Comisi\'on Nacional de Investigaci\'on
Cient\'{\i}fica y Tecnol\'ogica (CONICYT)},
and by the {\em Programa `Reticulados y
Ecuaciones'} of the Universidad de Talca.}

\date{June 2010}

\begin{abstract}
We present an explicit difference operator diagonalized by the
Macdonald polynomials associated with an (arbitrary) admissible pair
of irreducible reduced crystallographic root systems. By the duality
symmetry, this gives rise to an explicit Pieri formula for the
Macdonald polynomials in question. The simplest examples of our
construction recover Macdonald's celebrated difference operators and associated Pieri
formulas pertaining to the minuscule and quasi-minuscule weights.
As further by-products, explicit expansions and Littlewood-Richardson type formulas are obtained for the Macdonald polynomials associated with a special class of small weights.
\end{abstract}

\maketitle

\section{Introduction}\label{sec1}
In a landmark paper, Macdonald introduced his nowadays widespread families of basic hypergeometric orthogonal polynomials associated with (admissible pairs
of) crystallographic root systems \cite{mac:orthogonal}. It has been shown with the aid of the representation theory of double affine Hecke algebras that the polynomials in question form the joint eigenbasis of an algebra of commuting basic hypergeometric difference operators,
which is isomorphic to the Weyl-group invariant part of the group algebra over the weight lattice
\cite{mac:affine}. 
Even though for the simplest difference operators in the algebra at issue elegant explicit expressions can already be found in Macdonald's original work, to date it remains a challenging open problem to derive analogous explicit formulas providing
a complete system of generators for this algebra valid uniformly for all root systems
\cite[Sec. 4.4]{mac:affine}.

In this paper we present a compact explicit formula for a generalization of Macdonald's celebrated difference operators from \cite{mac:orthogonal}. Upon specialization, this formula gives rise to a complete system of generators in the case of all classical root systems and to various novel explicit independent generators in the case of the exceptional root systems.
In view of the remarkable duality symmetry enjoyed by the Macdonald polynomials \cite{mac:affine}, our difference operator formula also entails a new explicit Pieri-type
recurrence formula generalizing the explicit Pieri formulas known from Macdonald's theory (see e.g. \cite[Sec. 3]{las:inverse}).
One important reason motivating the quest for compact explicit Pieri formulas is that in the special case of the type $A$ root systems they have proven to be a powerful tool for deriving closed formulas for the expansion coefficients of Macdonald polynomials in terms of the standard bases of elementary and complete symmetric functions \cite{las-sch:inversion} (see also
\cite{die-lap-mor:determinantal}  and \cite{ste:graded} for alternative methods to compute Macdonald polynomials explicitly based on direct use of the difference equations and orthogonality relations, respectively). Indeed, as a further spin-off we also arrive at explicit expansions and Littlewood-Richardson type formulas for the Macdonald polynomials associated with a special class of small weights.

The paper is structured as follows. In Section \ref{sec2} the definition of the Macdonald polynomials is recalled. In Sections \ref{sec3} and \ref{sec4} we introduce our generalized Macdonald difference operator and apply the duality symmetry to derive the corresponding Pieri formula and its aforementioned by-products. 
Some of the technical aspects of the proof of our main result (Theorem \ref{diagonalization:thm} below) are relegated to an appendix  at the end of the paper.

\section{Macdonald polynomials}\label{sec2}
Let $E$ denote a real
finite-dimensional  Euclidean vector space with inner product $\langle \cdot ,\cdot \rangle$ and let $(R,S)$ be an admissible pair of  irreducible reduced root systems spanning $E$. (So $S=R$ or $S=R^\vee$, where $R^\vee:=\{\alpha^\vee\mid \alpha\in R\}$ with $\alpha^\vee:=2\alpha /\langle \alpha ,\alpha\rangle$.) 
We have a natural bijection $\alpha\mapsto \alpha_*$ of $R$ onto $S$ given by
 $\alpha_*:=u_\alpha^{-1}\alpha $ with $u_\alpha : =1$ if $S=R$ and $u_\alpha := \langle \alpha ,\alpha\rangle /2$ if $S=R^\vee$.
Let us denote the root lattice and the weight lattice of $R$ by $Q$ and $P$, respectively, and  let us write $Q^+$ for the semigroup generated by a (fixed) choice of positive roots $R^+$  and $P^+$  for the corresponding cone of dominant weights. The weight lattice is partially ordered by the dominance ordering, viz. for $\lambda ,\mu\in P$:
$\lambda\geq \mu$ if and only if $\lambda-\mu\in Q^+$.

The standard basis for the group algebra
$\mathbb{C}[P]$ is given by formal exponentials
$e^\lambda$, $\lambda\in P$ characterized by the relations $e^0=1$, 
$e^\lambda e^\mu=e^{\lambda +\mu}$. For $f=\sum_{\lambda\in P} f_\lambda e^\lambda$ in $\mathbb{C}[P]$ we define $\int f := f_0$ and $\overline{f}:=\sum_{\lambda\in P}\overline{f}_\lambda e^{-\lambda}$ (where $\overline{f}_\lambda$ refers to the complex conjugate of $f_\lambda$).  
Let $0<q<1$ and let $t:R\cup R^\vee \rightarrow (0,1)$ denote a root multiplicity function such that $t_{\alpha^\vee}=t_\alpha$ and $t_{w\alpha}=t_\alpha$ for all $w$ in the Weyl group $W$.
We will use the notation $q_\alpha:=q^{u_\alpha}$ and $(a;q)_m :=\prod_{k=0}^{m-1} (1-aq^k)$ with $m$ nonnegative integral or $\infty$.
For $f, g\in \mathbb{C}[P]$ the Macdonald inner product is now defined as \cite{mac:orthogonal,mac:affine}
\begin{subequations}
\begin{equation}\label{ipa}
\langle f,g\rangle_\Delta = |W|^{-1} \int  f \overline{g} \Delta ,
\end{equation}
where $|W|$ refers to the order of $W$ and
\begin{equation}\label{ipb}
\Delta := \prod_{\alpha \in R} \frac{( e^\alpha;q_\alpha)_\infty}{( t_\alpha e^\alpha;q_\alpha)_\infty} .
\end{equation}
\end{subequations}
In general (i.e. when $t_\alpha$ is {\em not} a positive integral power of $q_\alpha$), the product $\Delta$ \eqref{ipb} is not an element of $\mathbb{C}[P]$, however, it follows from Macdonald's theory that the RHS of \eqref{ipa} still makes sense as the constant term of a formal power series and, moreover, thus defined provides a nondegenerate positive inner product on  $\mathbb{C}[P]$ \cite[Sec. 5.1]{mac:affine}.

\begin{definition}[Macdonald Polynomials \cite{mac:orthogonal,mac:affine}]
For $\lambda\in P^+$ let
$m_\lambda := \sum_{\eta \in W\lambda} e^\eta$ (where the sum is meant over the $W$ orbit of $\lambda$). 
The Macdonald polynomials $p_\lambda$, $\lambda\in P^+$ (associated with the pair $(R,S)$) are defined as the unique polynomials in $\mathbb{C}[P]$ of the form
\begin{subequations}
\begin{equation}\label{mp-d1}
p_\lambda = m_\lambda + \sum_{\mu\in P^+,\, \mu < \lambda } a_{\lambda \mu} (q,t) m_\mu
\end{equation}
with expansion coefficients $a_{\lambda \mu} (q,t) $ such that
  \begin{equation}\label{mp-d2}
  \langle p_\lambda , m_\mu \rangle_\Delta =0 \quad \text{for all} \ \mu\in P^+ \ \text{with}\ \mu < \lambda  .
  \end{equation}
 \end{subequations}
\end{definition}

It is immediate from the definition that the Macdonald polynomials are orthogonal for distinct dominant weights that are comparable in the dominance ordering. Remarkably, one has in fact that \cite{mac:orthogonal}
\begin{equation}
  \langle p_\lambda , p_\mu \rangle_\Delta =0 \quad \text{for all} \ \lambda ,\mu\in P^+ \ \text{with}\ \lambda\neq \mu  .
  \end{equation}
In other words, the Macdonald polynomials $p_\lambda$, $\lambda \in P^+$ constitute an orthogonal basis of the $W$-invariant subalgebra $\mathbb{C}[P]^W$ of the group algebra with respect to the Macdonald inner product \eqref{ipa}, \eqref{ipb}
(where the Weyl group acts on $\mathbb{C}[P]$ via $we^\lambda :=e^{w\lambda}$, $w\in W$).

\section{The generalized Macdonald operator}\label{sec3}

\begin{definition}[Small Weight]
A dominant weight $\omega $ of a root system will be called {\em small} if $\langle \omega,\alpha^\vee\rangle \leq 2$ for any positive root $\alpha$.
\end{definition}
Well-known special cases of the small weights are the minuscule weights, i.e. dominant weights $\omega$ such that $\langle \omega,\alpha^\vee\rangle \leq 1$ for any positive root $\alpha$, and the quasi-minuscule weight, i.e. the weight $\omega=\alpha_0$ with $\alpha_0^\vee$ being the maximal root of the dual root system (in which case the maximum $\langle \omega,\alpha^\vee\rangle= 2$ is assumed only once, viz. when $\alpha = \omega$) \cite{bou:groupes}.
More generally, a weight $\omega$ is small if either: $\omega =0$, or $\omega$ is the sum of two (not necessarily distinct) minuscule weights, or $\omega$ is a small fundamental weight (cf. Table \ref{smallweights:tab}).
In this section we will introduce---for any small weight $\omega$ of $S^\vee$--- a difference operator $D_\omega:\mathbb{C}[P]^W\rightarrow \mathbb{C}[P]^W$ acting diagonally on the Macdonald basis.

Let $\Lambda$ denote the weight lattice of $S^\vee$ and let $\Lambda^+\subset \Lambda$ be the dominant cone (corresponding to our fixed choice of the positive roots). For $\omega\in \Lambda^+$ we define the saturated set
\begin{equation}\label{sat:eq}
\Lambda_\omega := \bigcup_{\mu\in\Lambda^+,\, \mu\leq \omega} W\mu\subset\Lambda
\end{equation}
(where the comparison of $\mu$ and $\omega$ is with respect to the dominance order of $S^\vee$) and
for $x\in E$ we define the translation operator
$T_x:\mathbb{C}[P]\rightarrow \mathbb{C}[P]$  by $T_x (e^\lambda):=q^{\langle \lambda ,x\rangle }e^\lambda$. 
For $\omega$ small our generalized Macdonald operator is now given by
\begin{subequations}
\begin{equation}\label{gmoa}
D_\omega:=\sum_{\nu\in\Lambda_\omega}\sum_{\eta\in W_\nu(w_\nu^{-1}\omega)} U_{\nu,\eta} V_\nu T_\nu ,
\end{equation}
where $w_\nu\in W$ is such that $w_\nu\nu\in\Lambda^+$, $W_\nu:=\{ w\in W \mid w\nu = \nu \}$, and
\begin{eqnarray}
&& V_\nu:=\prod_{\substack{\alpha\in R\\ \langle\alpha_*,\nu\rangle>0}} v_\alpha(e^\alpha)
      \prod_{\substack{\alpha\in R\\ \langle\alpha_*,\nu\rangle=2}} v_\alpha(q_\alpha e^\alpha)  ,\\
&& U_{\nu,\eta}:=\prod_{\substack{\alpha\in R_\nu \\ \langle\alpha_*,\eta\rangle>0}} v_\alpha(e^\alpha)
\prod_{\substack{\alpha\in R_\nu \\ \langle\alpha_*,\eta\rangle=2}} v_\alpha\bigl(q_\alpha^{-1} e^{-\alpha}\bigr) , \label{gmoc}
\end{eqnarray}
\end{subequations}
with
$v_\alpha (z) :=t_\alpha^{-1/2}(1-t_\alpha z)/(1-z)$. Here 
we have employed the notation $R_x:=  \{ \alpha\in R \mid \langle \alpha ,x \rangle =0\}$  for the subsystem of $R$ corresponding to the stabilizer $W_x$ of $x\in E$.

For  $x\in E$ and a full lattice $L\subset E$ (in our case $L=P$ or $L=\Lambda$), we define the evaluation homomorfism from $\mathbb{C}[L]$ into $\mathbb{C}$ by $e^\lambda (x):=q^{\langle\lambda ,x\rangle}$, $\lambda\in L$. The  main result of this paper is the following theorem.

\begin{theorem}[Diagonalization]\label{diagonalization:thm}
For $\omega\in \Lambda^+$ small and $\lambda\in P^+$, one has that
\begin{subequations}
\begin{equation}\label{ev:eq}
D_{\omega}p_\lambda=E_\omega^S (\lambda+\rho_{t,S}) p_\lambda,
\end{equation}
where 
\begin{eqnarray}
&& E_\omega^S:=\sum_{\mu\in\Lambda^+,\, \mu\leq \omega}  \epsilon_{\omega ,\mu} m_\mu\in\mathbb{C}[\Lambda], \\
&& \epsilon_{\omega ,\mu}:= \sum_{\eta\in W_\mu\omega}
\prod_{\substack{\alpha\in S_\mu^+ \\  |\langle\alpha  ,\eta\rangle |=1}} t_\alpha^{\langle\alpha,\eta\rangle}  
\end{eqnarray}
($S^+_\mu :=  \{ \alpha\in S^+ \mid \langle \alpha ,\mu \rangle =0\}$), and
\begin{equation}
\rho_{t,S} :=\frac{1}{2}\sum_{\alpha\in S^+}\log_q(t_\alpha) \alpha .
\end{equation}
\end{subequations}
\end{theorem}
\begin{proof}
Once stated the result, its proof is relatively straightforward following the original ideas of Macdonald from \cite{mac:orthogonal}. Specifically, the verification consists of three technical steps: (i) to check that the generalized Macdonald operator $D_\omega$ maps the space $\mathbb{C}[P]^W$ into itself, (ii) to check that the operator in question is triangular with respect to the basis of symmetric monomials $m_\lambda$, $\lambda\in P^+$ with corresponding eigenvalues given by $E^S_\omega (\lambda + \rho_{t,S})$,  $\lambda\in P^+$, and finally, (iii) to check that $D_\omega$ is symmetric with respect to the inner product \eqref{ipa}, \eqref{ipb}.
An outline of the details involved in the verification of each step is provided in the appendix at the end of the paper. From (i), (ii) and (iii) it follows that $E^S_\omega (\lambda + \rho_{t,S})^{-1}D_\omega p_\lambda$ satisfies the defining properties \eqref{mp-d1}, \eqref{mp-d2} of the Macdonald polynomial  $p_\lambda$, which proves the theorem.
\end{proof}

The operator $D_\omega$ maps $\mathbb{C}[P]$ into its quotient field.  However, it is immediate from its diagonal action on the Macdonald basis that the operator maps the subalgebra $\mathbb{C}[P]^W$ into itself (cf. the first step of the above proof). The simultaneous diagonalization by the Macdonald basis moreover implies the commutativity of the operators $D_\omega:\mathbb{C}[P]^W\rightarrow \mathbb{C}[P]^W$ corresponding to distinct choices of the small weight $\omega$.
\begin{corollary}[Commutativity]\label{commutativity:thm}
Let $\omega$ and $\tilde{\omega}$ be small weights of $S^\vee$. Then the corresponding generalized Macdonald operators $D_\omega:\mathbb{C}[P]^W\rightarrow \mathbb{C}[P]^W$ and $D_{\tilde{\omega}}:\mathbb{C}[P]^W\rightarrow \mathbb{C}[P]^W$ commute.
\end{corollary}
It is well-known that this type of commutativity in fact carries over to the commutativity of the difference operators in question viewed as operators in the quotient field of $\mathbb{C}[P]$.

When $\omega$ is minuscule $\Lambda_\omega=W\omega$.  The operator $D_\omega$ reduces in this case to Macdonald's difference operator \cite[\S 5]{mac:orthogonal}
\begin{subequations}
\begin{equation}\label{macop:m}
D_\omega=\sum_{\nu\in W \omega} V_\nu T_\nu  .
\end{equation}
When $\omega$ is quasi-minuscule $\Lambda_\omega=W\omega\cup\{ 0\}$. The operator $D_\omega$ then reduces to
\begin{equation}\label{macop:qm}
D_\omega=\sum_{\nu\in W\omega}(V_\nu T_\nu +U_{0,\nu}) ,
\end{equation}
\end{subequations}
which differs from Macdonald's difference operator \cite[\S 6]{mac:orthogonal} by an additive constant of the form
$\sum_{\nu\in W\omega} (V_\nu+U_{0,\nu})=\epsilon_{\omega ,0}+m_\omega (\rho_{t,S})$ (as Macdonald's operator in  \cite[\S 6]{mac:orthogonal}  annihilates the constant polynomial $p_0=1$).

By varying $\omega$ over the small weights of $S^\vee$, our formula $D_\omega$ \eqref{gmoa}-\eqref{gmoc} provides various explicit generators for the algebra of commuting difference operators diagonalized by the Macdonald polynomials \cite{mac:affine}. To avoid algebraic dependencies it is enough to consider only the generators corresponding to the small {\em fundamental} weights of $S^\vee$. The following lemma clarifies the structure of the saturated set $\Lambda_\omega$ \eqref{sat:eq} in this situation.

\begin{lemma}\label{small-weights1:lem}
Let $\omega$ be a small weight and let $\mu$ be a dominant weight such that $\mu \leq \omega$. Then  $\mu$ is small. Furthermore, if $\omega$ is moreover fundamental then  $\mu$ is fundamental or $\mu =0$.
\end{lemma}

The first part of this lemma is straightforward as $\langle \mu, \alpha^\vee\rangle\leq
\langle \mu, \alpha_0^\vee\rangle\leq \langle \omega, \alpha_0^\vee\rangle \leq 2$
for any positive root $\alpha$. The second part of the lemma is readily verified on a case by case basis by listing the small weights for all irreducible reduced root systems (but it would nevertheless be interesting to find a uniform classification-free proof of this fact in the spirit of \cite{ste:partial}).
Table \ref{smallweights:tab} provides a list of all small fundamental weights ordered by the dominance ordering. With the information in this table the construction of the generators by means of our formula for $D_\omega$ \eqref{gmoa}-\eqref{gmoc} becomes completely explicit. The number of independent generators found this way is given by the number of small fundamental weights. For the classical root systems all fundamental weights are small, so we cover a complete system of $n:=\text{Rank}(R)$ generators. The number of explicit independent generators that our formula fails to cover for the exceptional root systems is respectively: $1$ ($R=E_6$), $3$ ($R=E_7$), $6$ ($R=E_8$), $2$ ($R=F_4$), $1$ ($R=G_2$). For completeness the table also displays the number of ordered chains, which corresponds to the number of previously known explicit generators covered by Macdonald's difference operators associated with the minuscule and quasi-minuscule weights. Only for the root system $R=G_2$ no progress is made, as in this case the only nontrivial small weight is actually the quasi-minuscule one. Fortunately, in this situation the missing explicit formula for a second independent generator (violating the structure of $D_\omega$) is already known \cite{die-ito:difference}.

\begin{table}[hbt]
\vspace{2ex}
\begin{tabular}{|c|c|c|c|}
\hline
$R$ &small fundamental weights  &  \# weights & \# chains  \\ \hline\hline
$A_n$ &$\omega_1,\, \omega_2,\ldots ,\,\omega_n$  & $n$ &  $n$  \\ \hline
$B_n$ & $(0<)\omega_1<\omega_2<\cdots <\omega_{n-1},\,\omega_n$ & $n$ & $2$  \\ \hline
$C_n$ & $\begin{array}{c} \omega_1<\omega_3<\cdots <\omega_{2[\frac{n-1}{2}]+1} \\
(0<)\omega_2<\omega_4<\cdots <\omega_{2[\frac{n}{2}]} \end{array}$
 & $n$ & $2$  \\ \hline
$D_n$ &  $\begin{array}{c} \omega_1<\omega_3<\cdots <\omega_{2[\frac{n-1}{2}]-1} \\
(0<)\omega_2<\omega_4<\cdots <\omega_{2[\frac{n}{2}]-2}  \\ \omega_{n-1},\, \omega_n \end{array}$ & $n$ & $4$ \\ \hline
$E_6$ &   $\omega_1<\omega_5,\, (0<)\omega_2,\, \omega_6<\omega_3$ & $5$ &  $3$ \\ \hline
$E_7$ & $(0<)\omega_1<\omega_6,\,\omega_7<\omega_2$ & $4$ & $2$ \\ \hline
$E_8$ & $(0<)\omega_8<\omega_1$ & $2$ &  $1$ \\ \hline
$F_4$ &  $(0<)\omega_4<\omega_1$ & $2$ &  $1$  \\ \hline
$G_2$ &   $(0<)\omega_1$ & $1$ &  $1$ \\ \hline
\end{tabular}
\vspace{2ex}
\caption{Ordered chains of small fundamental weights, numbered in accordance with the conventions of the tables in Bourbaki \cite{bou:groupes}. 
(The comparison with the zero weight indicates that the lowest fundamental weight in the chain is quasi-minuscule rather than minuscule. By $[\frac{n-1}{2}]$ and $[\frac{n}{2}]$  we refer to the integral parts obtained by truncation.)} \label{smallweights:tab}
\end{table}

\section{Associated Pieri formula}\label{sec4}
It is well-known  that difference equations for the Macdonald polynomials immediately give rise to  Pieri-type recurrence relations by the remarkable duality symmetry \cite{mac:affine}. In order to exhibit the Pieri formulas associated with the difference equations of Theorem \ref{diagonalization:thm} explicitly, it will be convenient at this point to make the dependence of the construction on the admissible pair $(R,S)$ manifest by attaching it to all relevant objects. With this convention,
the renormalized Macdonald polynomial
\begin{subequations}
\begin{equation}
P_{\lambda}^{(R,S)}:= c_\lambda^{(R,S)} p_\lambda^{(R,S)} ,
\end{equation}
with
\begin{equation}
c_\lambda^{(R,S)} := q^{\langle \lambda,\rho_{t,R^\vee} \rangle}
\prod_{\alpha\in R^+}\frac{(q_\alpha^{\langle \alpha^\vee , \rho_{t,S}\rangle };   q_\alpha)_{\langle \alpha^\vee ,\lambda\rangle}}{(t_\alpha q_\alpha^{\langle \alpha^\vee ,\rho_{t,S}\rangle};q_\alpha)_{\langle \alpha^\vee,\lambda\rangle}} ,
\end{equation}
\end{subequations}
satisfies the  {\em Duality Symmetry} \cite[Sec. 5.3]{mac:affine}
 \begin{equation}\label{duality:eq}
P_{\lambda}^{(R,S)}(\mu+\rho_{t,R^\vee})=
   P_{\mu}^{(S^\vee,R^\vee)}(\lambda+\rho_{t,S}),\quad \text{for any} \;\lambda\in P^+,\mu\in\Lambda^+
\end{equation}
(so 
$P_{\lambda}^{(R,S)}(\rho_{t,R^\vee})=1$ in this normalization).

Combination of Theorem \ref{diagonalization:thm} with Eq. \eqref{duality:eq} entails the following recurrence relation.
\begin{theorem}[Pieri Formula]\label{pieri:thm}
For $\lambda, \omega\in P^+$ with $\omega$ small, one has that
\begin{eqnarray}\label{pieri:eq}
\lefteqn{E^{R^\vee}_\omega P_{\lambda}^{(R,S)}
=} && \\
&& \sum_{\substack{\nu\in P_\omega\\\lambda+\nu\in P^+}}\sum_{\eta\in W_\nu(w_\nu^{-1}\omega)}
 U^{(S^\vee,R^\vee)}_{\nu,\eta}(\lambda+\rho_{t,S}) V^{(S^\vee,R^\vee)}_\nu(\lambda+\rho_{t,S}) P_{\lambda+\nu}^{(R,S)}  ,\nonumber
\end{eqnarray}
where $P_\omega :=\bigcup_{\mu\in P^+,\mu\leq \omega} W\mu\subset P$, and $E^{R^\vee}$, 
$V^{(S^\vee,R^\vee)}_\nu$, $U^{(S^\vee,R^\vee)}_{\nu,\eta}$ are of the form in Theorem \ref{diagonalization:thm} (with the pair $(R,S)$ replaced by $(S^\vee,R^\vee)$).
\end{theorem}
\begin{proof}
Because the stated Pieri formula \eqref{pieri:eq} is a polynomial identity, it suffices to check that it is valid when evaluated at the points $x=\mu + \rho_{t,R^\vee}$ for all $\mu\in \Lambda^+$.
Upon evaluation at such a point and application of the
duality symmetry in Eq. \eqref{duality:eq}, the LHS goes over into $E^{R^\vee}_\omega ( \mu + \rho_{t,R^\vee}) P_{\mu}^{(S^\vee,R^\vee)} (\lambda + \rho_{t,S})$ and on the RHS $P_{\lambda+\nu}^{(R,S)}$ gets replaced by $P_{\mu}^{(S^\vee,R^\vee)} (\lambda +\nu + \rho_{t,S})$. We can now freely eliminate the restriction $\lambda +\nu\in P^+$ in the first sum on the RHS as the coefficient  $V^{(S^\vee,R^\vee)}_\nu (\lambda+\rho_{t,S})$ vanishes when this condition is not met. Indeed, this vanishing is caused
by the zero stemming from a factor of the type
$v_\beta(e^{-\beta})$ (if $\langle \beta_* ,\nu\rangle =-1$)
or  $v_\beta(e^{-\beta})v_\beta(q_\beta e^{-\beta})$ (if $\langle \beta_* ,\nu\rangle =-2$), with
$\beta $  a simple root of $S^\vee $ such that
$\langle \beta_*, \lambda +\nu\rangle <0$
(so $e^\beta (\lambda +\rho_{t,S})=t_\beta$ or $e^\beta (\lambda +\rho_{t,S})\in \{ t_\beta , q_\beta t_\beta\}$, respectively).
We thus end up with the evaluation at 
$x=\lambda + \rho_{t,S}$ of the equation
$E^{R^\vee}_\omega ( \mu + \rho_{t,R^\vee}) P_{\mu}^{(S^\vee,R^\vee)} = D_\omega^{(S^\vee,R^\vee)}P_{\mu}^{(S^\vee,R^\vee)} $, which holds by Theorem \ref{diagonalization:thm}.
\end{proof}

For $\nu\in W\omega$ the coefficients in the Pieri formula \eqref{pieri:eq} (i.e. the coefficients corresponding to the highest-weight orbit) already follow from analogous computations in the context of the double affine Hecke algebra
\cite[Sec. 5.3]{mac:affine}. In particular, for $\omega$ minuscule and for $\omega$ quasi-minuscule Theorem \ref{pieri:thm} recovers the Pieri formulas associated with the difference eigenvalue equations for the Macdonald  operators \eqref{macop:m} and \eqref{macop:qm}, respectively (cf. also \cite[Sec. 3]{las:inverse}):
\begin{subequations}
\begin{equation}
m_\omega P_{\lambda}^{(R,S)}=
\sum_{\substack{\nu\in W\omega \\ \lambda+\nu\in P^+}} V^{(S^\vee,R^\vee)}_\nu(\lambda+\rho_{t,S}) P_{\lambda+\nu}^{(R,S)}  
\end{equation}
for $\omega\in P^+$ minuscule, and 
\begin{equation}
\bigl( m_\omega -m_\omega ( \rho_{t,R^\vee})\bigr) P_{\lambda}^{(R,S)}=
\sum_{\substack{\nu\in W\omega \\ \lambda+\nu\in P^+}} V^{(S^\vee,R^\vee)}_\nu(\lambda+\rho_{t,S}) \left( P_{\lambda+\nu}^{(R,S)}  -P_{\lambda}^{(R,S)} \right)
\end{equation}
\end{subequations}
for $\omega\in P^+$ quasi-minuscule.

With the aid of the Pieri formula, it is not difficult to compute the Macdonald polynomials explicitly for small weights.

\begin{theorem}[Macdonald polynomials for small weights]\label{macsw:thm}
For $\omega\in P^+$ small, the (monic) Macdonald polynomial is given by
\begin{equation}\label{expansion}
p_\omega^{(R,S)}=
\sum_{\substack{\mu_1<\cdots <\mu_\ell=\omega\\ \ell =1,\ldots ,n_\omega}}
(-1)^{\ell -1} 
E^{R^\vee}_{\mu_1}
\prod_{1\leq k\leq \ell -1}
U^{(S^\vee,R^\vee)}_{\mu_{k},\mu_{k+1}} (\rho_{t,S}) ,
\end{equation}
with $n_\omega := | P_\omega\cap P^+ |$ and where the summation is meant over (all strictly ascending chains of) the {\em dominant weights} only.
\end{theorem}
\begin{proof}
For $\lambda =0$ the Pieri formula \eqref{pieri:eq} reduces to
\begin{equation*}
E^{R^\vee}_\omega 
=
 \sum_{\mu\in P^+,\,\mu\leq \omega}
 U^{(S^\vee,R^\vee)}_{\mu,\omega}(\rho_{t,S})p_{\mu}^{(R,S)}  ,
\end{equation*}
where it was used that
$V^{(S^\vee,R^\vee)}_\mu(\rho_{t,S}) P_{\mu}^{(R,S)} =p_\mu^{(R,S)}$ for
$\mu$ small and that
the factor
$U^{(S^\vee,R^\vee)}_{\mu,\eta}(\rho_{t,S})$ with $\eta\in W_\mu (\omega )$ vanishes unless $\eta=\omega$ (cf. the proof of Theorem \ref{pieri:thm}).
The stated result now follows upon a unitriangular matrix inversion.
\end{proof}
For the nonreduced root systems a related expansion formula
was obtained in \cite[Thm. 5.1]{kom-nou-shi:kernel} and Pieri formulas can be found
in \cite[Sec. 6]{die:properties}. Combination of Theorems \ref{pieri:thm} and \ref{macsw:thm} immediately entails a Littlewood-Richardson type linearization formula for the product of two Macdonald polynomials with one of the weights small.
\begin{theorem}[Littlewood-Richardson type rule]
For $\lambda,\omega\in P^+$ with $\omega$ small, the product of the corresponding Macdonald polynomials expands as
\begin{eqnarray}
\lefteqn{P_{\omega}^{(R,S)}P_{\lambda}^{(R,S)}=
c_\omega^{(R,S)} \sum_{\substack{\mu_1<\cdots <\mu_\ell=\omega\\ \ell =1,\ldots ,n_\omega}}
(-1)^{\ell -1} 
\prod_{1\leq k\leq \ell -1}
U^{(S^\vee,R^\vee)}_{\mu_{k},\mu_{k+1}} (\rho_{t,S})  } &&  \\
&& \times \sum_{\substack{\nu\in P_{\mu_1}\\\lambda+\nu\in P^+}}\sum_{\eta\in W_\nu(w_\nu^{-1}\mu_1)}
 U^{(S^\vee,R^\vee)}_{\nu,\eta}(\lambda+\rho_{t,S}) V^{(S^\vee,R^\vee)}_\nu(\lambda+\rho_{t,S}) P_{\lambda+\nu}^{(R,S)} \nonumber
\end{eqnarray}
(adopting the notational conventions of Theorem \ref{macsw:thm}).
\end{theorem}
Much less compact Pieri formulas, explicit representations, and Littlewood-Richardson type rules for the Macdonald polynomials have surfaced recently in the context of work on explicit Pieri formulas for
Macdonald's spherical functions \cite{die-ems:pieri}.
Furthermore, for the classical root systems of type $C$ (inverse) Pieri formulas expressing the products of two one-row Macdonald polynomials in terms of two-row Macdonald polynomials (and vice versa) were
presented in \cite{las:inverse}, whereas in \cite{bar:pieri-type} a start was made with the study of Pieri formulas for the nonsymmetric Macdonald polynomials associated with the root systems of type $A$.

\appendix
\section{The diagonalization of $D_\omega$: technicalities}\label{app}
In this appendix we outline the technical details of the three principal steps underlying the proof of
Theorem \ref{diagonalization:thm}.

\subsection{Step (i): Polynomiality}
Our proof of the fact that $D_\omega$ maps $ \mathbb{C}[P]^W$ into itself hinges on the following
lemmas.
\begin{lemma}\label{small-weights2:lem}
Let $\nu,\omega\in \Lambda^+$ be small with $\nu\leq\omega$ and let $\beta\in R^+$ such that $\langle \beta_* ,\nu\rangle =2$. 
Then one has that:
\begin{itemize}
\item[a)] the set  $\{ \alpha\in R^+ \mid \langle \nu ,\alpha_*\rangle =2, \langle\alpha ,\alpha\rangle = \langle\beta ,\beta\rangle \}$ is given by the orbit $W_{\nu ,\omega} (\beta)$ (where $W_{\nu,\omega}  := W_{\nu} \cap W_{\omega} $);
\item[b)]  there exists a unique root $\alpha\in W_{\nu ,\omega} (\beta)\subset R^+$ such that $\underline{\nu}:=\nu -\alpha_*^\vee $ is dominant (and thus $\underline{\nu}$ is small by Lemma \ref{small-weights1:lem}).
\end{itemize}
\end{lemma}
\begin{lemma}\label{small-weights3:lem}
Let $\nu,\omega\in \Lambda^+$ be small with $\nu<\omega$ and let $\beta\in R^+_{\nu}$ such that $\langle \beta_* ,\omega\rangle =2$. 
Then one has that:
\begin{itemize}
\item[a)] the set  $\{ \alpha\in R^+_{\nu} \mid \langle \omega ,\alpha_*\rangle =2, \langle\alpha ,\alpha\rangle = \langle\beta ,\beta\rangle \}$ is given by the orbit $W_{\nu,\omega} (\beta)$;
\item[b)]  there exists a unique root $\alpha\in W_{\nu,\omega}   (\beta)\subset R^+_\nu$ such that $\overline{\nu} :=\nu +\alpha_*^\vee $ is dominant (and moreover $\overline{\nu} \leq \omega$ so $\overline{\nu}$ is small by Lemma \ref{small-weights1:lem}).
\end{itemize}
\end{lemma}
We verified these lemmas on a case by case basis through a list of the small weights for all irreducible reduced root systems $R$ (without loss of generality one may assume for this check that $S=R^\vee$). Nevertheless, in part b) of Lemma \ref{small-weights2:lem} the existence of a positive root
$\alpha$ such that $\nu -\alpha_*^\vee $ is dominant  is already immediate from \cite[Corollary 2.7]{ste:partial}  and it looks in fact promising to try to pursue a classification-free proof of these lemmas by building on the analysis  in {\em loc. cit.}

\begin{proposition}\label{step1:prp}
For $\omega\in \Lambda^+$ small and  any $\lambda\in P^+$,  one has that $$D_\omega m_\lambda \in \mathbb{C}[P]^W.$$
\end{proposition}
\begin{proof}
Clearly $D_\omega m_\lambda$ constitutes a Weyl-group invariant element in the quotient field
of $\mathbb{C}[P]$. To demonstrate that this element lies in fact in the group algebra itself, we will check that $(D_\omega m_\lambda) (x)$ is regular as a function of $x$ (in the complexification of $E$).  Indeed, in the terms of the corresponding evaluated expression
$$
(D_\omega m_\lambda) (x)=\sum_{\nu\in\Lambda_\omega}\sum_{\eta\in W_\nu(w_\nu^{-1}\omega)} U_{\nu,\eta} (x)V_\nu (x) m_\lambda (x+\nu ) 
$$
generically simple poles appear at hyperplanes congruent to $\langle x, \beta_* \rangle =0$ or  $1+\langle x, \beta_* \rangle = 0$ for some $\beta \in R$ 
(due to the zero in the denominator of $v_\beta (z)$ at $z=1$). The Weyl-group symmetry ensures that residues of the poles congruent to $\langle x, \beta_* \rangle =0$ cancel out in $(D_\omega m_\lambda) (x)$ and that, moreover,  for the poles  congruent to $1+\langle x, \beta_* \rangle =0$ it is sufficient to check the vanishing of the residues originating from terms with $\nu$  dominant and $\eta=\omega$. For such $\nu, \eta$ the relevant poles in the corresponding term $U_{\nu,\omega} (x)V_\nu (x) m_\lambda (x+\nu )$ arise (i) in $V_\nu (x)$ at 
$1+\langle x, \beta_* \rangle =0$ with $\beta \in R^+$ such that $\langle \beta_* ,\nu\rangle =2$
and (ii) in $U_{\nu,\omega} (x)$ at $1+\langle x, \beta_* \rangle =0$ with $\beta \in R^+_\nu$ such that $\langle \beta_* ,\omega \rangle =2$. (In particular, type (i) poles only occur when $\nu\neq 0$ and type (ii) poles only occur when $\nu <\omega$.) We will now show that in $(D_\omega m_\lambda) (x)$ residues of the type (i) cancel pairwise against residues of the type (ii). Indeed, exploiting the invariance of the term in question with respect to the stabilizer $W_{\nu ,\omega}$ we restrict attention to the type (i) pole(s) corresponding the $\alpha\in R^+$ with $\underline{\nu}:=\nu-\alpha_*^\vee$ dominant (cf. Lemma \ref{small-weights2:lem}) and to the type (ii) pole(s)
corresponding to the $\alpha \in R^+_\nu$ with  $\overline{\nu} :=\nu +\alpha_*^\vee \leq \omega$ dominant (cf. Lemma \ref{small-weights3:lem}). 
The type (i) residue in $U_{\nu,\omega} (x)V_\nu (x) m_\lambda (x+\nu )$ at
$1+\langle x, \alpha_* \rangle =0$ cancels against the corresponding type (ii) residue
in $U_{\underline{\nu},\omega} (x)V_{\underline{\nu}} (x) m_\lambda (x+\underline{\nu} )$, because
$$m_\lambda (x+\nu )\mid_{\langle x, \alpha_* \rangle =-1} = m_\lambda (x+\underline{\nu}) \mid_{\langle x, \alpha_* \rangle =-1}$$
by the Weyl-group invariance and furthermore
$$
\text{Res}\left[ 
\prod_{\substack{\beta\in R^+\\ \langle\beta_*,\nu\rangle>0}} v_\beta(e^\beta)
\prod_{\substack{\beta\in R^+\\ \langle\beta_*,\nu\rangle=2}} v_\beta(q_\beta e^\beta) 
\prod_{\substack{\beta\in R_\nu^+ \\ \langle\beta_*,\omega\rangle>0}} v_\beta(e^\beta)
\prod_{\substack{\beta\in R_\nu^+ \\ \langle\beta_*,\omega\rangle=2}} v_\beta\bigl(q_\beta^{-1} e^{-\beta}\bigr) 
\right]_{\langle x, \alpha_* \rangle =-1}
$$
equals
$$
-\text{Res}\left[ 
\prod_{\substack{\beta\in R^+\\ \langle\beta_*,\underline{\nu}\rangle>0}} v_\beta(e^\beta)
\prod_{\substack{\beta\in R^+\\ \langle\beta_*,\underline{\nu}\rangle=2}} v_\beta(q_\beta e^\beta) 
\prod_{\substack{\beta\in R_{\underline{\nu}}^+ \\ \langle\beta_*,\omega\rangle>0}} v_\beta(e^\beta)
\prod_{\substack{\beta\in R_{\underline{\nu}}^+ \\ \langle\beta_*,\omega \rangle=2}} v_\beta\bigl(q_\beta^{-1} e^{-\beta}\bigr) 
\right]_{\langle x, \alpha_* \rangle =-1} 
$$
(where the argument $x$ was suppressed for typographical reasons). To infer the above equality of the monomial symmetric function evaluations it is enough to observe that 
$x+\underline{\nu}=r_\alpha (x+\nu)$  when $x$ lies on the hyperplane $1+\langle x, \alpha_* \rangle =0$ (where $r_\alpha$ denotes the orthogonal reflection in the hyperplane $\langle x, \alpha \rangle =0$). Furthermore, to verify the corresponding relation between the residues it is convenient to divide both products and cancel common factors in the numerator and denominator upon exploiting the hyperplane relation  $1+\langle x, \alpha_* \rangle =0$. It thus follows after  straightforward but somewhat tedious
computations (performed on a case by case basis for all admissible pairs $(R,S)$ and all nonzero small weights
$\nu,\omega$ with $\nu\leq\omega$)
that  on the hyperplane in question the quotient collapses to $-1$  (as an equality between rational functions).
(Notice in this connection that $\text{Res}\, v_\alpha (q_\alpha z)\mid_{z=1/q_\alpha}=-\text{Res}\, v_\alpha (q_\alpha^{-1} z^{-1})\mid_{z=1/q_\alpha}$.)
Reversely,  the type (ii) residue in $U_{\nu,\omega} (x)V_\nu (x) m_\lambda (x+\nu )$ at
$1+\langle x, \alpha_* \rangle =0$ 
cancels against the corresponding type (i) residue
in $U_{\overline{\nu},\omega} (x)V_{\overline{\nu}} (x) m_\lambda (x+\overline{\nu} )$
(by the same residue formula with $\nu$ and $\underline{\nu}$ being replaced by $\overline{\nu}$ and $\nu$, respectively).
\end{proof}

\subsection{Step (ii): Triangularity and eigenvalues}

\begin{proposition}\label{step2:prp}
For $\omega\in \Lambda^+$ small and  any $\lambda\in P^+$,  one has that $$D_\omega m_\lambda = \sum_{\mu\in P^+,\, \mu\leq \lambda} b_{\lambda\mu}(q,t)m_\mu ,$$ with the leading expansion coefficient  (i.e. the corresponding eigenvalue) given by
$$b_{\lambda\lambda}(q,t)=  E_\omega^S (\lambda+\rho_{t,S}) .$$
\end{proposition}

\begin{proof}
By Proposition \ref{step1:prp}, $D_\omega m_\lambda$ expands as a finite linear combination of symmetric monomials $m_\mu$, $\mu\in P^+$.  Since $\lim_{z\to 0} v_\alpha (z)=t_\alpha^{-1/2}$ and 
$\lim_{z\to \infty} v_\alpha (z)=t_\alpha^{1/2}$, it readily follows that the asymptotics deep in the
anti-dominant Weyl chamber is of the form:
$(D_\omega m_\lambda) (x)= E_\omega^S (\lambda+\rho_{t,S}) q^{\langle \lambda ,x\rangle} (1+o(1))$ as $x$ grows away from the walls in such a way that  $\langle x, \alpha\rangle \to -\infty$ for all $\alpha \in R^+$. Hence, in the expansion of $D_\omega m_\lambda$  only monomials 
 $m_\mu$, $\mu\in P^+$ with $\mu \leq \lambda$ appear, and the coefficient of the leading monomial $m_\lambda$ is given by $E_\omega^S (\lambda+\rho_{t,S}) $.
\end{proof}

\subsection{Step (iii): Symmetry}

\begin{proposition}\label{step3:prp}
For $\omega\in \Lambda^+$ small and any $\lambda,\mu\in P^+$, one has that
$$\langle D_\omega m_\lambda ,m_\mu\rangle_\Delta=\langle  m_\lambda , D_{\omega} m_\mu\rangle_\Delta .$$
\end{proposition}

\begin{proof}
Let us for the moment assume that our root multiplicity function $t$ is such that $t_\alpha$ is an integral power $\geq 2$ of $q_\alpha$. In this situation all terms of $ (D_\omega m_\lambda) \Delta$ lie in $\mathbb{C}[P]$, and it is moreover seen that
\begin{eqnarray*}
&& \langle D_\omega m_\lambda ,m_\mu\rangle_\Delta \\
&& = |W|^{-1} \sum_{\nu\in\Lambda_\omega}\sum_{\eta\in W_\nu(w_\nu^{-1}\omega)} 
\int  U_{\nu,\eta} V_\nu (T_\nu  m_\lambda ) \overline{m_\mu}  \Delta \\
&& \stackrel{(*)}{=} |W|^{-1} \sum_{\nu\in\Lambda_\omega}\sum_{\eta\in W_\nu(w_\nu^{-1}\omega)} 
\int  m_\lambda \overline{(T_{\nu} U_{\nu,-\eta} V_{-\nu} m_\mu  \Delta )} \\
&& \stackrel{(**)}{=} |W|^{-1} \sum_{\nu\in\Lambda_\omega}\sum_{\eta\in W_\nu(w_\nu^{-1}\omega)} 
\int  m_\lambda \overline{U_{\nu,\eta} V_{\nu} (T_{\nu} m_\mu )}  \Delta \\
&& =\langle  m_\lambda , D_{\omega} m_\mu\rangle_\Delta .
\end{eqnarray*}
Here we have used  in Step (*) that
$\overline {V_\nu}=V_{-\nu}$, $\overline {U_{\nu ,\eta}}=U_{\nu,-\eta}$, $\overline{\Delta}=\Delta$, and $\int (T_\nu f)\overline{g}=\int f \overline{(T_{\nu} g)}$ for all $f,g\in \mathbb{C}[P]$,
and in Step (**) that  $T_{\nu} V_{-\nu}\Delta = V_{\nu}\Delta $,
$T_{\nu} U_{\nu,-\eta}= U_{\nu,-\eta}$, and
$$ \sum_{\eta\in W_\nu(w_\nu^{-1}\omega)} U_{\nu,-\eta}=   \sum_{\eta\in W_\nu(w_\nu^{-1}\omega)} U_{\nu,w_o \eta}  = \sum_{\eta\in W_\nu(w_\nu^{-1}\omega)} U_{\nu,\eta}  ,$$ 
where $w_o$ refers to the longest element of the Weyl group $W_\nu$.
This completes the proof of the proposition (and thus that of Theorem \ref{diagonalization:thm}) for
 $t_\alpha$ an integral power $\geq 2$ of $q_\alpha$.  Since both sides of the eigenvalue equation \eqref{ev:eq} are rational expressions in $t_\alpha$ (cf.  \cite[\S 4]{mac:orthogonal}),
this is in fact sufficient to conclude
that Theorem \ref{diagonalization:thm} (and thus the present proposition) holds
for general root multiplicity functions $t$.
\end{proof}

\vspace{3ex}
\noindent {\bf Acknowledgments.} The exploratory phase of our research,
culminating in the formula of Theorem \ref{diagonalization:thm}, has benefited a lot from explicit computer algebra computations involving Stembridge's Maple packages {\tt COXETER} and
{\tt WEYL}.

\bibliographystyle{amsplain}

\end{document}